\documentclass[a4paper,11pt]{amsart}
\usepackage[latin1]{inputenc}
\usepackage{amssymb,anysize}
\usepackage[cmtip,all]{xy}
\usepackage{hyperref}
\usepackage{cite}
\marginsize{3cm}{3cm}{2.5cm}{2.5cm}
\newcommand{\AAA}{\mathbb{A}}
\newcommand{\CC}{\mathbb{C}}
\newcommand{\kk}{\Bbbk}
\newcommand{\CCC}{C^{\bullet,\bullet}}
\newcommand{\RR}{\mathbf R}
\newcommand{\LL}{\mathbf L}
\newcommand{\DD}{\mathcal{D}}

\newcommand{\HH}{\mathcal{H}}

\newcommand{\OO}{\mathcal{O}}
\newcommand{\MM}{\mathcal{M}}
\newcommand{\NN}{\mathcal{N}}
\newcommand{\II}{\mathcal{I}}
\newcommand{\JJ}{\mathcal{J}}
\newcommand{\Acal}{\mathcal{A}}
\newcommand{\MV}{\operatorname{MV}_{\{Y_i\}}}
\newcommand{\Gamba}{\Gamma_{[\{Y_i\}]}}
\newcommand{\EI}{\mbox{ }^\text{I}E}
\newcommand{\EII}{\mbox{ }^\text{II}E}
\newcommand{\Mod}{\mathcal{M}od}
\newcommand{\Hom}{\mathcal{H}om}
\newcommand{\Cha}{\mathcal{C}ha}
\newcommand{\Tot}{\operatorname{Tot}}
\newcommand{\DR}{\operatorname{DR}}
\newcommand{\ra}{\rightarrow}
\newcommand{\lra}{\longrightarrow}
\newcommand{\hra}{\hookrightarrow}
\newcommand{\D}{\text{D}^\text{b}}
\newcommand{\Dc}{\text{D}_{\text{c}}^\text{b}}

{\theoremstyle{definition}
\newtheorem{defi}{Definition}[section]}
{\theoremstyle{remark}
\newtheorem{nota}[defi]{Remark}}
\newtheorem{teo}[defi]{Theorem}
\newtheorem{prop}[defi]{Proposition}
\newtheorem{coro}[defi]{Corollary}
\newtheorem{lema}[defi]{Lemma}

\begin{document}

\title{Two Mayer-Vietoris Spectral Sequences for $\mathcal{D}$-modules}
\author{Alberto Castaño Domínguez}
\thanks{The author is supported by FQM218, MTM2010-19298 and P12-FQM-2696.}
\email{albertocd@us.es}
\address{Departamento de Álgebra \& Instituto de Matemáticas de la Universidad de Sevilla (IMUS), Universidad de Sevilla. Edificio Celestino Mutis. Avda. Reina Mercedes s/n 41012 Sevilla}
\subjclass[2010]{Primary 14B15, 14F05, 14F10, 18G40}
\date{May 2014}

\begin{abstract}
We provide two Mayer-Vietoris-like spectral sequences related to the localization over the complement of a closed subvariety of an algebraic variety by using techniques from $\mathcal{D}$-modules and homological algebra. We also give, as an application of the previous, a method to calculate the cohomology of the complement of any arrangement of hyperplanes over an algebraically closed field of characteristic zero.
\end{abstract}

\maketitle

\section{Introduction}

Local cohomology and localization of sheaves of abelian groups have been of interest since the sixties, when Grothendieck introduced them in a seminar at Harvard (\hspace{-.5pt}\cite{Ha2}). Since then, they have become a common tool when working in algebraic geometry or commutative algebra, for they appear naturally when studying sheaf cohomology, $\DD$-modules, depth or cohomological dimension.

An algebraic variety, or just variety, will mean for us an equidimensional separated finite type scheme, reducible or not, over an algebraically closed field of characteristic zero. In this note we give two Mayer-Vietoris spectral sequences of the localization of certain $\OO_X$-modules over the open complement of a closed subvariety $Y=\bigcup_i Y_i$ of an algebraic variety $X$ over an algebraically closed field of characteristic zero. For a complex of $\OO_X$-modules $\MM\in\D(\OO_X)$, one can define the localization of $\MM$, denoted by $\RR\MM(*Y)$, as the image of $\MM$ by the right derived functor of $\varinjlim_k\HH om_{\OO_X}\left(\JJ_Y^k,\bullet\right)$, $\JJ_Y$ being the ideal of definition of $Y$. If $\MM$ is of quasi-coherent cohomology, Grothendieck's classical version and this one coincide. For this functor we prove in theorem \ref{MV} the existence of the spectral sequence of bounded complexes of quasi-coherent $\OO_X$-modules
$$E_1^{p,q}=\bigoplus_{|I|=1-p}\RR^q\MM(*Y_I)\Rightarrow_p\RR^{p+q}\MM(*Y),$$
where $Y_I$ is the intersection of the components (not necessarily irreducible) $Y_i$ for $i\in I$. This way of dividing $Y$ and taking the spectral sequence is completely analogous to how Àlvarez Montaner, García López and Zarzuela Armengou acted with local cohomology of modules (with support in certain ideals) in \cite{AGZ}, work which was generalized by Lyubeznik in \cite{Ly}.

As the title says, there is another spectral sequence provided in theorem \ref{MVrel}, very related to the one written above, but in a relative version. To achieve that, we work with $\DD_X$-modules, by using the direct image functor in the derived category of coherent $\DD$-modules associated with a morphism $f:X\ra Z$, denoted by $f_+$. The spectral sequence takes a complex of $\DD_X$-modules $\MM\in \Dc(\DD_X)$ and deals with complexes of $\DD_Z$-modules like this:
$$E_1^{p,q}=\bigoplus_{|I|=1-p}\HH^q f_+\RR\MM(*Y_I) \Rightarrow_p\HH^{p+q}f_+\RR\MM(*Y).$$

Despite the abundant presence of Mayer-Vietoris-like spectral sequences in the literature, we only found an analogue of the second one when $f$ is a projection over a point in \cite[Sommes trig. 2.6.2*]{SGA4}, but using $\ell$-adic cohomology with compact support.

The relative spectral sequence allows us to compute in a purely algebraic way the global algebraic de Rham cohomology of the complement of an (affine or projective) arrangement of hyperplanes over any algebraically closed field of characteristic zero. In the case it were $\CC$, by \cite[Theorem 1']{Gr} we know that the global algebraic de Rham cohomology of that complement is the same as its singular cohomology, giving in particular a proof of the well known result of Orlik and Solomon \cite[5.3]{OS}, whose original proof requires more background on the combinatorics of the intersection poset of the arrangement and its characteristic and Poincaré polynomials.

\vspace{.25cm}

\textbf{Acknowledgements.} The author wants to thank his doctoral advisors, Luis Narváez Macarro and Antonio Rojas León, the suggestion of the topic, the encouragement to tackle it and their careful reading of previous versions of this text.

\section{Basics on spectral sequences}

In this section we will recall some facts about spectral sequences that will be useful in the following. We will only work with cohomological spectral sequences, so that adjective will be omitted.

\begin{defi}
A spectral sequence in an abelian category $\Acal$ is a family $\{E_r^{p,q}\}$ of objects in $\Acal$ for every integers $p,q$ and for every integer $r\geq0$, such that for each $(p,q,r)$ there is a morphism, called differential, $d_r^{p,q}:E_r^{p,q}\ra E_r^{p+r,q-r+1}$ satisfying that $d_r^{p+r,q-r+1}\circ d_r^{p,q}=0$.

The subfamily of objects $E_r:=\{E_r^{p,q}\}$ for a fixed $r$ is called the $r$-th page, or sheet, of the spectral sequence, and we name the family of all differentials $d_r^{p,q}$ with $r$ fixed $d_r:E_r\ra E_r$. The chain condition for the $d_r^{p,q}$ can be written as $d_r^2=0$.

Moreover, we also have isomorphisms
$$\HH^{p,q}(E_r)=\ker d_r^{p,q}/\operatorname{im} d_r^{p-r,q+r-1}\stackrel\sim\lra E_{r+1}^{p,q}.$$
\end{defi}

\begin{defi}
Let $E=\{E_r^{p,q}\}$ be a spectral sequence such that for every $r\geq r(p,q)$, it holds that $E_r^{p,q}=E_{r(p,q)}^{p,q}$. We define the limit term of $E$ as $E_{\infty}^{p,q}:=E_{r(p,q)}^{p,q}$, and we say that $E$ abuts to $E_{\infty}$.
\end{defi}

The limit term of a spectral sequence is what gives us the desired information. There are some cases in which it exists and is easy to compute:

\begin{nota}
Let $E$ be a spectral sequence. If there exists a $r_0\geq0$ such that $d_r=0$ for every $r\geq r_0$, then $E_{r_0}=E_{\infty}$, for $E_{r+1}=\HH(E_r)=E_r$. In that case we say that $E$ degenerates at $r_0$.

Now suppose that there exists an $r_0\geq2$ such that $E_{r_0}$ is concentrated in a single row or column. Then we have that every differential $d_r^{p,q}$ departs from or arrives at the zero object, so the spectral sequence degenerates at the $r_0$-th page. In this special case of degeneration we say that the spectral sequence collapses at the $r_0$-th sheet.
\end{nota}

\begin{defi}
Let $E$ be a spectral sequence. It is said to converge if there exists a graded object $H^{\bullet}$, with a finite filtration $F^{\bullet}H^{\bullet}$, such that the limit term of $E$ is the graded complex associated to $F^{\bullet}$, that is,
$$E_{\infty}^{p,q}=G^pH^{p+q}=F^pH^{p+q}/F^{p+1}H^{p+q}.$$
We denote this by $E_r^{p,q}\Rightarrow_pH^{p+q}$.
\end{defi}

This is what spectral sequences are for; they usually allow us to calculate an approximation by means of a filtration of an interesting filtrated object hard to deal with, by computing some other objects in a simpler way.

For instance, if $E$ is a spectral sequence collapsing at the $s$-th page, it converges to $H^{\bullet}$, where $H^n$ is the only $E_s^{p,q}\neq0$ such that $p+q=n$.

We are going to introduce a special kind of spectral sequences that will be of help in the following: the spectral sequences of a double complex. Recall that a double complex in $\Acal$ is a bigraded complex $\CCC$ with differentials $d_I^{p,q}:C^{p,q}\ra C^{p+1,q}$ and $d_{II}:C^{p,q}\ra C^{p,q+1}$ such that $d_I^2=d_{II}^2=d_Id_{II}+d_{II}d_I=0$.

\begin{nota}
With each complex of complexes $\mathbf{C}=\left(C^{\bullet}\right)^{\bullet}$ we can associate a bicomplex in an obvious way just by taking as vertical differentials those of $\mathbf{C}$ and horizontal differentials the ones of $\mathbf{C}$ multiplied by $(-1)^q$ in the $q$-th row.
\end{nota}

\begin{defi}
Let $\CCC$ be a double complex. Its total complex, $\Tot(C)^{\bullet}$, is the complex given by
$$\Tot(C)^n=\bigoplus_{p+q=n}C^{p,q},$$
with differentials $d_T$ given by $d_T=d_I+d_{II}$. It can be endowed with two filtrations, the horizontal and vertical ones, given respectively by
$$F_I^p(\Tot(C)^n)=\bigoplus_{r+s=n,r\leq p}C^{r,s} \text{ and } F_{II}^p(\Tot(C)^n)=\bigoplus_{r+s=n,s\leq p}C^{r,s}.$$
\end{defi}

\begin{prop}
Let $\CCC$ be a double complex. Then, there exist two spectral sequences, called usual, $\EI$ and $\EII$, given by
$$\EI_0^{p,q}=\EII_0^{p,q}=C^{p,q} \text{ and } \EI_1^{p,q}=\HH^p(C^{\bullet,q});\,\EII_1^{p,q}=\HH^q(C^{p,\bullet}).$$
If the bicomplex $\CCC$ can be translated to occupy either the first or the third quadrant, both spectral sequences converge to the cohomology of the total complex, that is,
$$\EI_{\infty}^{p,q}\Rightarrow_p\HH^{p+q}(Tot(C)^{\bullet}) \text{ and } \EII_{\infty}^{p,q}\Rightarrow_p \HH^{p+q}(Tot(C)^{\bullet}).$$
\end{prop}
\begin{proof}
Take into account that if we translate to the first or third quadrant our complex, we do not change the structure of its associated usual spectral sequences, so we can assume that it lies directly on one of those quadrants and then apply \cite[11.17]{Ro}.
\end{proof}

A complex having a finite number of nonvanishing and left or right bounded rows or columns fulfills the condition of the proposition. Note that although both spectral sequences have a grading of the total complex as limit term, they do not need to be the same, since the filtrations that induce them are different.

Spectral sequences arising from double complexes appear very frequently, but this is not the only way to obtain a spectral sequence. Two further constructions are the spectral sequences associated with an exact couple or a filtered complex. See, for example, \cite[\S~11]{Ro} for more information.

\section{Mayer-Vietoris spectral sequence}

For any variety $Z$, we will denote by $\pi_Z$ the projection from $Z$ to a point. In what follows, $X$ will denote a smooth algebraic variety, and $Y\subseteq X$ will be a closed subvariety of $X$ defined by the ideal $\JJ_Y$. Whenever we talk about a complex of $\OO_X$- or $\DD_X$-modules, we will understand them as objects of the corresponding derived category of bounded complexes, which will be clear from the context.

After \cite[Remark 5]{Gr}, we can define the functor $\bullet(*Y)$ of $\Mod(\OO_X)$ given by
$$\MM(*Y):=\varinjlim_k\HH om_{\OO_X}\left(\JJ_Y^k,\MM\right).$$

\begin{nota}
Let $\II^{\bullet}$ be an acyclic complex of injective $\OO_X$-modules. Since $\HH om_{\OO_X}(\bullet,\II^q)$ is an exact functor for every $q$, the complex $\HH om_{\OO_X}(\JJ_Y^k,\II^{\bullet})$ will be acyclic for every $k$, and so will be $\II(*Y)$ because direct limits commute with cohomology as long as it is an exact functor. Therefore, by \cite[I.5.1]{Ha1}, the functor $\bullet(*Y)$ is left exact and can be right derived to provide a functor
$$\RR\bullet(*Y): \D(\OO_X)\lra \D(\OO_X).$$
\end{nota}

\begin{nota}\label{acyc}
Let $j:X-Y\hra X$ denote the open immersion from the complement of $Y$ into $X$, and let us define (cf. \cite[I.6.1]{Me}) the algebraic local cohomology of an $\OO_X$-module $\MM$ as
$$\RR^i\Gamma_{[Y]}(\MM):=\varinjlim_k\RR^i\HH om_{\OO_X} \left(\OO_X/\JJ_Y^k,\MM\right).$$
Because of the same reason as above, $\Gamma_{[Y]}$ is a left exact functor. From the exact sequence $0\ra\JJ_Y^k\ra\OO_X\ra\OO_X/\JJ_Y^k\ra0$ and \cite[Corollary 1.9, 2.8]{Ha2}, we obtain a commutative diagram
$$\xymatrix{ 0 \ar[r] & \Gamma(Y,\MM) \ar[r] & \MM \ar[r] & j_*j^{-1}\MM \ar[r] & \RR^1\Gamma(Y,\MM) \ar[r] &0\\
0 \ar[r] & \Gamma_{[Y]}(\MM) \ar[r] \ar[u] & \MM \ar[r] \ar[u] & \MM(*Y) \ar[r] & \RR^1\Gamma_{[Y]}(\MM) \ar[r] \ar[u] &0},$$
where the first and fourth objects of the top and the bottom row are, respectively, the first two local cohomology modules of $\MM$ over $Y$ and their algebraic counterparts (cf. \cite{Ha2}). Then we have a morphism $\MM(*Y)\ra j_*j^{-1}\MM$, which, again by \cite[2.8]{Ha2}, becomes an isomorphism if $\MM$ is of quasi-coherent cohomology, as well as with $\RR\MM(*Y)\ra\RR j_*j^{-1}\MM$.

As a consequence, for every quasi-coherent injective $\OO_X$-module $\II$, we have that $\II(*Y)=j_*j^{-1}\II$ is another quasi-coherent injective $\OO_X$-module by \cite[1.4.10]{EGA}.
\end{nota}

From now on, let us assume that $Y$ can be decomposed as the union of $r$ different closed subvarieties $Y_i\subseteq X$, $i=1,\ldots,r$. For each $I\subseteq\{1,\ldots,r\}$, we will write $Y_I=\bigcap_{i\in I}Y_i$. If $I=\emptyset$, $Y_I=Y$.

\begin{defi}
We define the functor $\MV:\Mod(\OO_X)\lra\mathcal{C}(\OO_X)$ given by
$$\MV^p(\MM)=\left\{\begin{array}{cl}
\displaystyle\bigoplus_{|I|=1-p}\MM(*Y_I) & p=-(r-1),\ldots,0\\
0 & \text{otherwise}\end{array}\right.,$$
with connecting morphisms consisting of an alternating sum of the canonical morphisms $\rho_{I,J}:\MM(*Y_I)\ra\MM(*Y_J)$ whenever $I\supset J$ induced by the inclusions of the respective ideals of definition, $\eta_{J,I}:\JJ_{Y_J}\hra\JJ_{Y_I}$. More precisely, if we denote by $I_j$ the subset resulting of taking out of $I$ its $j$-th element,
$$\begin{array}{rcl} \displaystyle\bigoplus_{|I|=1-p}\MM(*Y_I)& \lra& \displaystyle \bigoplus_{|J|=-p}\MM(*Y_J)\\
\alpha_I&\longmapsto & \displaystyle \bigoplus_{j=0}^{-p}(-1)^j\rho_{I,I_j}\left(\alpha_I \right)\end{array}.$$
It is straightforward to see that these morphisms make $\MV(\MM)$ into a complex.

Any morphism between two $\OO_X$-modules $\MM$ and $\NN$ gives rise to a morphism between $\MM(*T)$ and $\NN(*T)$ for every closed subvariety $T\subset X$, just by applying the corresponding hom functor and taking direct limits. Thus the image by $\MV$ of a morphism $\MM\ra\NN$ is just the chain map consisting of the direct sum of their associated morphisms at every degree.
\end{defi}

\begin{prop}\label{exacta}
Let $\II$ be an injective $\OO_X$-module. Then the complex $\MV(\II)$ is exact except in degree zero, in which its cohomology is $\II(*Y)$.
\end{prop}
\begin{proof}
To prove this statement we will introduce two complexes. Let us define $\Gamba(\MM)$ to be the complex defined by
$$\Gamba^p(\MM)=\left\{\begin{array}{cl}
\displaystyle\bigoplus_{|I|=1-p}\Gamma_{[Y_I]}(\MM) & p=-(r-1),\ldots,0\\
0 & \text{otherwise}\end{array}\right.,$$
with morphisms given by
$$\begin{array}{rcl} \displaystyle\bigoplus_{|I|=1-p}\Gamma_{[Y_I]}(\MM)& \lra& \displaystyle \bigoplus_{|J|=-p}\Gamma_{[Y_J]}(\MM)\\
\alpha_I&\longmapsto & \displaystyle \bigoplus_{j=0}^{-p}(-1)^j\rho_{I,I_j}^L\left(\alpha_I \right)\end{array}$$
as chain maps, $\rho_{I,I_j}^L$ being the morphisms associated with the canonical inclusions $\eta_{J,I}:\JJ_{Y_J}\hra\JJ_{Y_I}$ for $J\subseteq I$. As with $\MV$, it can easily be proved that it is a complex.

The other complex that we will provide, denoted by $\Cha(\MM)$, mimics this behaviour of $\Gamba(\bullet)$ and $\MV(\bullet)$, but taking as objects just copies of $\MM$. Namely,
$$\Cha^p(\MM)=\left\{\begin{array}{cl}
\displaystyle\bigoplus_{|I|=1-p}\MM & p=-(r-1),\ldots,0\\
0 & \text{otherwise}\end{array}\right..$$
The chain maps are just alternating sums of identity morphisms as with the other two complexes.

Now for every injective $\OO_X$-module $\II$, we can form an exact sequence
$$0\lra\Gamba(\II)\lra\Cha(\II)\lra\MV(\II)\lra0,$$
where, at each index, we take the exact sequence induced by applying direct sums, direct limits and the exact functor (since $\II$ is injective) $\HH om_{\OO_X}(\bullet,\II)$ to
$$0\lra\JJ_{Y_I}^k\lra\OO_X\lra\OO_X/\JJ_{Y_I}^k\lra0.$$

Thanks to \cite[2.1]{Ly} we know that, for every $x\in X$, $\Gamba(\II)_x$ is exact except at degree zero, in which its cohomology is $\Gamma_{[Y]}(\II)_x$. On the other hand, $\Cha(\II)_x$ is just the simplicial complex of cohomology associated with the standard $(r-1)$-simplex $\Delta^{r-1}$ with coefficients in the abelian group $\II_x$. Consequently, its $p$-th cohomology will vanish but for $p=0$, being $\II_x$ at that point.

Thus if we take stalks at $x$ on our exact sequence of complexes and form its long exact sequence of cohomology, we can deduce that at every $x\in X$ the cohomology of $\MV(\II)_x$ vanishes everywhere except in zero degree, being there $\II_x/\Gamma_{[Y]}(\II)_x\cong\II(*Y)_x$.

Having the same for every stalk, we can go upstairs to $X$ thanks to \cite[2.6]{Iv} and obtain what we wanted to prove.
\end{proof}

Once we have settled that important fact that we will use in the following, we can state our main result in this section.

\begin{teo}\label{MV}
For every $\MM\in \text{\emph{D}}_\text{\emph{qc}}^{\text{\emph{b}}}(\OO_X)$, there exists a spectral sequence of the form
$$E_1^{p,q}=\bigoplus_{|I|=1-p}\RR^q\MM(*Y_I)\Rightarrow_p\RR^{p+q}\MM(*Y).$$
\end{teo}
\begin{proof}
Let us take a quasi-coherent $\OO_X$-injective resolution $\II$ of $\MM$ (we can do it thanks to \cite[II.7.18]{Ha1}), and form the double complex $\CCC$, given by $C^{p,q}=\MV^p(\II^q)$, with vertical differentials given by the images by the functor $\MV^p$ of the ones of $\II^\bullet$, and horizontal differentials those of $\MV^\bullet(\II^q)$ multiplied by $(-1)^q$.

Since $C^{\bullet,\bullet}$ occupies the first quadrant (and $r$ nonzero columns), its usual spectral sequences will converge to the cohomology of the total complex, $\HH^n\left(\Tot(\CCC)\right)$.

The first sheet of the first of those usual spectral sequences is, by proposition \ref{exacta},
$$\EI_1^{p,q}=\HH^p\left(\MV^\bullet(\II^q)\right)=\left\{
\begin{array}{cl}
\II^q(*Y)& p=0\\
0 & \text{otherwise}\end{array}\right.$$
Now, since the second page of this spectral sequence is the vertical cohomology of the first one and the latter is concentrated in one column, we have that
$$\EI_2^{p,q}=\HH^q\left(\HH^p\left(\MV^\bullet(\II^q)\right)\right)=\left\{
\begin{array}{cl}
\RR^q\MM(*Y)& p=0\\
0 & \text{otherwise}\end{array}\right.,$$
so $\EI_r$ collapses and $\HH^n\left(\Tot(\CCC)\right)=\EI_2^{0,n}=\RR^n\MM(*Y)$.

On the other hand, the first page of the other usual spectral sequence is given by $\EII_1^{p,q}=\HH^q\left(C^{p,\bullet}\right)$. In our context, we have by definition that
$$\EII_1^{p,q}=\HH^q\left(\MV^p(\II^{\bullet})\right)= \bigoplus_{|I|=1-p}\RR^q\MM(*Y_I).$$
Since $\EII_1^{p,q}\Rightarrow_p\RR^{p+q}\MM(*Y)$, we obtain what we wanted to prove.
\end{proof}

Note that when $r=1$ the spectral sequence is trivial and gives no additional information. When $r=2$ we have several short exact sequences of the form
$$0\lra E_{\infty}^{-1,n+1}\lra\RR^n\MM(*Y)\lra E_{\infty}^{0,n}\lra0,$$
so in this case we already obtain a different (and more detailed) information than by using the Mayer-Vietoris long exact sequence \cite[I.6.2]{Me}.

\section{Relative Mayer-Vietoris spectral sequence}

In this section we will present a relative version of the above mentioned spectral sequence, but for $\DD_X$-modules, by using the direct image functor for them.

\begin{defi}
Let $f:X\ra Y$ be a morphism of smooth varieties. The direct image of $\DD_X$-modules is the functor $f_+:\D(\DD_X)\ra \D(\DD_Y)$ given by
$$f_+\MM:=\RR f_*\left(\DD_{Y\leftarrow X}\otimes_{\DD_X}^{\LL}\MM\right),$$
where $\DD_{Y\leftarrow X}$ is the $\left(f^{-1}\DD_Y,\DD_X\right)$-bimodule
$$\DD_{Y\leftarrow X}:=\omega_X\otimes_{f^{-1}\OO_Y}f^{-1}\Hom_{\OO_{Y}} \left(\omega_Y,\DD_Y\right) ,$$
called the transfer $\DD$-module for the direct image of $f$. In the formula, $\omega_X$ is the right $\DD_X$-module of top differential forms on $X$.
\end{defi}

\begin{nota}\label{directa}
When $f:U\hra X$ is an open immersion, $f_+=\RR f_*$, because $\DD_{X\leftarrow U}\cong f^{-1}\DD_X=\DD_U$.

When $f:X=Y\times Z\ra Z$ is a projection, $\DD_{Z\leftarrow X}\otimes_{\DD_X}^{\LL}\MM$ is nothing but a shifting by $\dim Y$ places to the left of the relative de Rham complex of $\MM$
$$\DR_f(\MM):=0\lra\MM\lra\MM\otimes_{\OO_{X}}\Omega_{X/Z}^1\lra\ldots\lra \MM\otimes_{\OO_{X}}\Omega_{X/Z}^n\lra0,$$
so we will have that $f_+\cong\RR f_*\DR_f(\bullet)[\dim Y]$ (\hspace{-.5pt}\cite[I.5.2.2]{Me}). When $Z$ is a point, the functor $\RR f_*$ is just the derived global sections functor $\RR\Gamma(X,\bullet)$, and in that special case the functor $f_+$ is just a shifting of global de Rham cohomology.
\end{nota}

Let us introduce now another important image functor in $\DD$-module theory.

\begin{defi}
Let $f:X\ra Y$ be a morphism of smooth varieties. The inverse image of $\DD_X$-modules is the functor $f^+:\D(\DD_Y)\ra \D(\DD_X)$ given by
$$f^+\MM:=\DD_{X\ra Y}\otimes_{f^{-1}\DD_Y}^{\LL}f^{-1}\MM,$$
where $\DD_{X\ra Y}$ is the $\left(\DD_X,f^{-1}\DD_Y\right)$-bimodule
$$\DD_{X\ra Y}:=\OO_X\otimes_{f^{-1}\OO_Y} f^{-1}\DD_Y,$$
called the transfer $\DD$-module for the inverse image of $f$.
\end{defi}

\begin{nota}\label{inversa}
Just by substituting the expression of $\DD_{X\ra Y}$ into the formula for $f^+$ we see that the inverse image of $\DD_X$-modules coincides with the derived inverse image of $\OO_X$-modules, $\LL f^*\bullet=\OO_X\otimes_{f^{-1}\OO_Y}^{\LL} f^{-1}\bullet$. Then, if $f$ is a flat morphism, $f^+=f^*$. In the special case in which $f:U\hra X$ is an open immersion, $f^+=f^{-1}$.
\end{nota}

\begin{teo}\label{MVrel}
Let $f:X\lra Z$ be a morphism between two smooth algebraic varieties and let $Y=\bigcup_i Y_i$ a closed subvariety of $X$. Then, for every $\MM\in \text{\emph{D}}_\text{\emph{c}}^{\text{\emph{b}}}(\DD_X)$, there exists a spectral sequence of complexes of $\DD_Z$-modules of the form
$$E_1^{p,q}=\bigoplus_{|I|=1-p}\HH^q f_+\RR\MM(*Y_I) \Rightarrow_p\HH^{p+q}f_+\RR\MM(*Y).$$
\end{teo}
\begin{proof}
First take into account that every morphism can be decomposed as a closed immersion into its graph followed by the canonical projection over the second component, so if we prove that for any closed immersion $i:X\lra Z$ we have that $i_+\RR\MM(*Y)\cong\RR(i_+\MM)(*Y)$, we will only need to prove the statement of the theorem in the case in which $f=\pi:X=T\times Z\lra Z$ is a projection.

Indeed, consider the cartesian diagram given by
$$\xymatrixcolsep{3pc}\xymatrix{\ar @{} [dr] |{\Box} X-Y \ar[d]_{\bar{i}} \ar[r]^{j} & X \ar[d]^{i}\\
 Z-Y \ar[r]^{\bar{j}} & Z}.$$
We know that $\MM$ is a coherent $\DD_X$-module, hence quasi-coherent $\OO_X$-module, so $\RR\MM(*Y)\cong j_+j^+\MM$. By the smooth base change theorem \cite[1.7.3]{HTT},
$$i_+j_+j^+\MM=\bar{j}_+\bar{i}_+ j^+\MM\cong\bar{j}_+\bar{j}^+i_+\MM.$$
Now $i_+\MM$ is a quasi-coherent $\OO_Z$-module (\hspace{-.5pt}\cite[1.5.24]{HTT}), whence $\bar{j}_+\bar{j}^+i_+\MM\cong\RR(i_+\MM)(*Y)$ and we are done.

Thus assume that $f$ is a projection as in the first paragraph. For every $I\subset\{1,\ldots,r\}$, let us define $U_I=X-Y_I$ and denote by $j_I$ the open immersion of $U_I$ into $X$, and define also $j_0:U_0:=X-Y\hra X$. Since $\MM$ is of coherent cohomology over $\DD_X$, it is of quasi-coherent cohomology over $\OO_X$, and by virtue of remarks \ref{acyc}, \ref{directa} and \ref{inversa}, $\RR\MM(*Y_I)\cong\RR j_{I,*}j_I^{-1}\MM\cong j_{I,+}j_I^+\MM$. Therefore we will have that $\pi_+\RR\MM(*Y_I)=\pi_+j_{I,+}j_I^+\MM=(\pi\circ j_I)_+j_I^+\MM$, by \cite[1.5.21]{HTT}. As a consequence,
$$\pi_+\RR\MM(*Y_I)=\RR(\pi\circ j_I)_*\left(\DD_{Z\leftarrow U_I}\otimes_{\DD_{U_I}}^{\LL}j_I^{-1}\MM\right).$$
Now take into account that $\DD_{U_I}=j_I^{-1}\DD_X$ and $\DD_{Z\leftarrow U_I}=j_I^{-1}\DD_{Z\leftarrow X}$, so we can write $\pi_+\RR\MM(*Y_I)$ as
$$\RR (\pi\circ j_I)_*j_I^{-1}\left(\DD_{Z\leftarrow X}\otimes_{\DD_X}^{\LL}\MM\right).$$
The analogous result holds for $\pi_+\RR\MM(*Y)\cong\RR (\pi\circ j_0)_*j_0^{-1}\left(\DD_{Z\leftarrow X}\otimes_{\DD_X}^{\LL}\MM\right)$.

Recall that $\DD_{Z\leftarrow X}\otimes_{\DD_X}^{\LL}\MM\cong\DR_\pi(\MM)[\operatorname{codim}_XZ]$ because of $\pi$ being a projection. $\DR_\pi(\MM)$ does not belong to the category of complexes of quasi-coherent $\OO_X$-modules because its chain maps are just linear over our field of definition; however, it is a complex in the category of sheaves of abelian groups whose objects are quasi-coherent $\OO_X$-modules. This slight difference allows us to take an injective Cartan-Eilenberg resolution of it in the category of sheaves of abelian groups, but having injective quasi-coherent $\OO_X$-module as objects. To see this, just note that in the dual of the proof of \cite[5.7.2]{We} every (classical) injective resolution that we form can be taken within the category of quasi-coherent $\OO_X$-modules. The problem appears when one has to lift linear maps, since it cannot provide a morphism of $\OO_X$-modules. Nevertheless, this drawback can be controlled because chain morphisms do not affect the properties of the objects, and taking the total complex of that Cartan-Eilenberg resolution, we turn out to have an injective resolution $\II^{\bullet}$ of $\DR_\pi(\MM)[\operatorname{codim}_XZ]$ in the category of sheaves of abelian groups whose objects are much more than that, since they are quasi-coherent $\OO_X$-modules.

Consequently, let us build the bicomplex $\CCC$ with objects
$$C^{p,q}=\bigoplus_{|I|=1-p}(\pi\circ j_I)_*j_I^{-1}\II^q=\pi_*\MV^p(\II^q),$$
where the last equality holds because of our careful choice of $\II^{\bullet}$, being the vertical and horizontal differentials the image by $\pi_*$ of those from $\MV^p(\II^{\bullet})$ and the differentials of $\MV^{\bullet}(\II^q)$ multiplied by $(-1)^q$, respectively.

As in the proof of theorem \ref{MV}, we will take the usual spectral sequences for that double complex, which has $r$ bounded below nonvanishing columns. Then those spectral sequences will converge to the cohomology of the total complex associated with $\CCC$.

Since $\pi_*$ is a left exact functor and the $\II^q(*Y_I)$ are acyclic, the first usual spectral sequence has as first page
$$\EI_1^{p,q}=\HH^p\left(C^{\bullet,q}\right)=\left\{\begin{array}{cl}
\pi_*\II^q(*Y) & p=0\\
0 & \text{otherwise}\end{array}\right.$$
This is because we were working with horizontal differentials, which are $\OO_X$-linear. Therefore the second sheet of this spectral sequence will be
$$\EI_2^{p,q}=\HH^q\left(\HH^p\left(C^{\bullet,q}\right)\right)\cong \left\{\begin{array}{cl}
\RR^q(\pi\circ j_0)_*j_0^{-1}\DR_\pi(\MM)[\dim T] & p=0\\
0 & \text{otherwise}\end{array}\right.$$
As it happened in the proof of theorem \ref{MV}, this spectral sequence collapses, and in consequence
$$\HH^n\left(\Tot(\CCC)\right)=\EI_2^{0,n}\cong\HH^n\pi_+ \RR\MM(*Y).$$
Note that the last isomorphism is just a consequence of having the isomorphism $\DD_{Z\leftarrow X}\otimes_{\DD_X}^{\LL}\MM\cong\DR_\pi(\MM)[\dim T]$ with complexes of quasi-coherent $\OO_X$-modules as objects.

Let us see what expression the other usual spectral sequence has. Its first page is the vertical cohomology of the double complex, that is to say,
$$\EII_1^{p,q}=\HH^q(C^{p,\bullet})\cong\bigoplus_{|I|=1-p}\RR^q (\pi\circ j_I)_*j_I^{-1}\II^q\cong\bigoplus_{|I|=1-p} \HH^q\pi_+\RR\MM(*Y_I).$$
There is no objection to that; what we only needed were kernels and cokernels, and they are the same in both senses.

In conclusion,
$$E_1^{p,q}=\bigoplus_{|I|=1-p} \HH^q\pi_+\RR\MM(*Y_I)\Rightarrow_p \HH^{p+q}\pi_+ \RR\MM(*Y),$$
as desired.
\end{proof}

\section{Arrangements of hyperplanes}

Now we will exemplify the usefulness of theorem \ref{MVrel} with the calculation of the global de Rham cohomology of the complement of an arrangement of hyperplanes $\Acal$ over an algebraically closed field $\kk$ of characteristic zero. As we will see, it is much influenced by the combinatorics of its intersection poset. We will need a special kind of global Künneth formula, stated as following:

\begin{lema}
For any two smooth varieties $X$ and $Y$, we have that
$$\pi_{X\times Y,+}\OO_{X\times Y}\cong\pi_{X,+}\OO_X\otimes_{\kk}\pi_{Y,+}\OO_Y.$$
\end{lema}
\begin{proof}
Let us consider the following cartesian diagram:
$$\xymatrix{\ar @{} [dr] |{\Box} X\times Y \ar[d]_{p_1} \ar[r]^{p_2} & Y \ar[d]^{\pi_Y}\\
 X \ar[r]^{\pi_X} & \{*\}}.$$
Obviously, $\pi_{X\times Y,+}\OO_{X\times Y}\cong\pi_{X,+}p_{1,+}p_2^+\OO_Y$, so by the smooth base change theorem we have that
$$\pi_{X\times Y,+}\OO_{X\times Y}\cong\pi_{X,+}\pi_X^+\pi_{Y,+}\OO_Y\cong \pi_{X,+}\left(\OO_X\otimes_{\OO_X}\pi_X^+\pi_{Y,+}\OO_Y\right)\cong \pi_{X,+}\OO_X\otimes_{\kk}\pi_{Y,+}\OO_Y,$$
where the last isomorphism is given by the projection formula \cite[1.7.5]{HTT}.
\end{proof}

Let us return to arrangements. Even though the arrangement is projective, its complement is still affine, since we can consider one of the hyperplanes as the one at infinity, so we will formulate the result just for affine arrangements. Thus let $X=\AAA^n$ and $Y$ be the subvariety of $\AAA^n$ given by the union of the hyperplanes of $\Acal$, which we will rename to $Y_1,\ldots,Y_r$. Let $\MM=\OO_{\AAA^n}$. Denote by $\pi_{Z}$ the projection to a point from a variety $Z$. We have the spectral sequence
$$E_1^{p,q}=\bigoplus_{|I|=1-p}\HH^q\pi_{\AAA^n,+}\RR\OO_{\AAA^n}(*Y_I) \Rightarrow_p\HH^{p+q}\pi_{\AAA^n,+}\RR\OO_{\AAA^n}(*Y).$$
If $\Acal$ is not essential, let us denote its rank by $r<n$ and the variety formed by the essential arrangement associated with $\Acal$ by $Y'$. Then $\AAA^n-Y\cong(\AAA^r-Y')\times\AAA^{n-r}$, so by the global Künneth formula we know that $\pi_{\AAA^n,+}\RR\OO_{\AAA^n}(*Y)\cong\pi_{\AAA^r,+}\RR\OO_{\AAA^r}(*Y')[n-r]$.

In order to use the spectral sequence, we must know all of the $\pi_{\AAA^n,+}\RR\OO_{\AAA^n}(*Y_I)$, for which we need to do a little work. Recall that for every closed subvariety $T\subset X$ and every $\MM\in \Dc(\DD_X)$, we have the isomorphism $\RR\MM(*T)\cong j_+j^+\MM$, $j$ being the open immersion of the complement of $T$ into $X$. Moreover, we can form the triangle in $\D\left(\DD_{X}\right)$
$$\RR\Gamma_{[T]}(\MM)\lra \MM\lra j_+j^+\MM=\RR\MM(*T)\lra,$$
associated with the diagram $X-T\stackrel{j}{\ra}X\stackrel{i}{\leftarrow}T$, and if $T$ is smooth we can replace $\RR\Gamma_{[T]}(\MM)$ by $i_+i^+\MM[-\operatorname{codim}_XT]$ (cf. \cite[1.7.1]{HTT}).

\begin{prop}
Let $Y$ be the variety formed by the union of the hyperplanes $Y_i$, with $i=1,\ldots,r$, of an essential affine arrangement over an algebraically closed field of characteristic zero $\kk$.

For any pair of integers $(p,q)$, let
$$d_{p,q}=\operatorname{card}\{\emptyset\neq I\subseteq \{1,\ldots,r\}\,|\,|I|=1-p,\,\dim Y_I=(n-q-1)/2\},$$
and let $p_{q,0}$ and $p_{q,1}$ be, for a fixed $q\neq-n$, the least and greatest $p$, if any, such that $d_{p,q}\neq0$, respectively. Then, for any $i=-n+1,\ldots,0$ there exists a unique integer $q$ such that $q+p_{q,1}=i$, and
$$\dim\HH^i\pi_{\AAA^n,+}\RR\OO_{\AAA^n}(*Y)= (-1)^{-p_{q,1}}\sum_{p=p_{q,0}}^{p_{q,1}}(-1)^pd_{p,q},$$
If $i=-n$, $\HH^i\pi_{\AAA^n,+}\RR\OO_{\AAA^n}(*Y)=\kk$, and if $i\notin\{-n,\ldots,0\}$, $\HH^i\pi_{\AAA^n,+}\RR\OO_{\AAA^n}(*Y)$ vanishes.
\end{prop}
\begin{proof}
Note that the last statement follows from the fact that $\pi_{\AAA^n,+}$ is nothing but taking global de Rham cohomology, shifted $n$ places to the left, and we are dealing with affine varieties and quasi-coherent $\OO$-modules.

For every $I\subset\{1,\ldots,r\}$ let $r_I=\dim Y_I$. We have that $\AAA^n-Y_I\cong\left(\AAA^{n-r_I}-\{\underline{0}\}\right)\times\AAA^{r_I}$, so we only need, by virtue of the global Künneth formula above, to compute the global de Rham cohomology of the affine space $\AAA^m$ minus one point for every $m$. In order to do that we can use the excision triangle as above with $T=\{\underline{0}\}$, namely
$$i_+i^+\OO_{\AAA^m}[-m]\lra\OO_{\AAA^m}\lra \RR\OO_{\AAA^m}(*\{\underline{0}\})\lra.$$
Applying the direct image functor associated with the projection $\pi_{\AAA^m}$ we get another triangle of graded $\kk$-vector spaces
$$\kk[-m]\lra\kk[m]\lra\pi_{\AAA^m,+} \RR\OO_{\AAA^m}(*\{\underline{0}\}),$$
so $\pi_{\AAA^m,+}\RR\OO_{\AAA^m}(*\{\underline{0}\})=\kk[m] \oplus\kk[-m+1]$.

Thus for every $I\subset\{1,\ldots,r\}$ we have that
$$\pi_{\AAA^n,+}\RR\OO_{\AAA^n}(*Y_I)=\left\{\begin{array}{cl}
\kk[n]\oplus\kk[-n+2r_I+1] & \text{if }Y_I\neq\emptyset\\
\kk[n] & \text{if }Y_I=\emptyset\end{array}\right..$$

By definition, the first page of our relative Mayer-Vietoris spectral sequence is
$$E_1^{p,q}=\left\{\begin{array}{cl}
\kk^{\binom{r}{1-p}}& p=-(r-1),\ldots,0 \text{ , }q=-n\\
\kk^{d_{p,q}}& q\neq-n \text{ and } d_{p,q}\neq0\\
0& \text{otherwise}\end{array}\right..$$
For a fixed $q$, the differentials between the $E_1^{p,q}$ terms are induced by the differentials in $\MV(\II)$ for an injective $\OO_{\AAA^n}$-module $\II$, whose cohomologies vanished except in degree zero, and so will happen with $E_1^{\bullet,q}$.

Whenever we have an exact sequence of vector spaces of the form
$$V:0\lra V_0\lra\ldots\lra V_s,$$
the dimension of the last cohomology (that is, the $s$-th one), is
$$\dim \operatorname{coker}(V_{s-1}\ra V_s)=(-1)^s\sum_{i=0}^s(-1)^i\dim V_i,$$
which can be easily proved by induction. Then, when $q=-n$, the dimension of the last cohomology space of $E_1^{\bullet,-n}$ is
$$\sum_{i=1}^r\binom{r}{i}(-1)^{i-1}=-\left((1-1)^r-1\right)=1,$$
while for other $q$ such that $d_{p,q}\neq0$ for some $p$, the dimension of the last cohomology space of $E_1^{\bullet,q}$ is
$$e_{p,q}:=(-1)^{-p_{q,1}}\sum_{p=p_{q,0}}^{p_{q,1}}(-1)^pd_{p,q},$$
vanishing otherwise.

Thus we can affirm that at the second sheet of our spectral sequence,
$$\dim E_2^{p,q}=\left\{\begin{array}{cl}
e_{p,q} & \text{if }q\neq -n, p=p_{q,1} \text{ and } d_{p,q}\neq0\\
1 & \text{if }p=0, q=-n\\
0& \text{otherwise}\end{array}\right..$$
By definition, apart from $q=-n$, $E_2^{p,q}=0$ if $q-n$ is even. It is easy to see that $E_2=E_{\infty}$, for any $d_r^{p,q}$ maps $E_r^{p,q}$ to $E_r^{p+r,q-r+1}$, and for no $r$ we can go from one point of the form $(p,n+2k-1)$ neither to another $(p',n+2k'-1)$ nor to $(0,-n)$, for any couple of integers $k$ and $k'$, so $E_{2}^{p,q}=E_{\infty}^{p,q}\neq0$. Furthermore, note that $p_{q,1}=(1-q-n)/2$, because it is one minus the least amount of distinct hyperplanes which suffice to intersect in a variety of dimension $r=(n-q-1)/2$, which is $n-r$ (we are using here that the arrangement is essential), so for each integer $i=-n,\ldots,0$ there is at most just one pair $(p,q)$ satisfying $p=p_{q,1}$ (when it can be defined) and $p+q=i$. Summing up, our spectral sequence degenerates at the second page and
$$\dim\HH^i\pi_+\RR\OO_{\AAA^n}(*Y)=\left\{\begin{array}{cl}
e_{p,q} & \text{if }e_{p,q}\neq 0 \text{ for }i=p+q\\
1 & \text{if } i=-n\\
0 & \text{otherwise}\end{array}\right..$$
\end{proof}

\begin{coro}\label{arrangposgen}
Under the same assumptions of the proposition, if $\Acal$ is an affine arrangement in general position,
$$\pi_+\RR\OO_{\AAA^n}(*Y)=\bigoplus_{i=-n}^0\kk^{\binom{r}{i+n}}[-i],$$
where $\binom{a}{b}=0$ if $a<b$ by convention.
\end{coro}

Recall that the proposition reproduces in our context the decomposition given by Orlik and Solomon in \cite[5.3]{OS}, when $\Acal$ is affine and central. Although we can reduce to this case for other kind of arrangements, our result gives a more direct proof of the general case.


\begin{thebibliography}{SGA 4 1/2}
\bibitem[AGZ]{AGZ} J. Àlvarez Montaner, R. García López, S. Zarzuela Armengou, Local cohomology, arrangements of subspaces and monomial ideals. \textit{Adv. Math.} \textbf{174} (2003), no. 1, 35-56.
\bibitem[Gr]{Gr} A. Grothendieck, On the de Rham cohomology of algebraic varieties. \textit{Inst. Hautes Études Sci. Publ. Math.}, \textbf{29} (1966), 95-103.
\bibitem[Ha1]{Ha1} R. Hartshorne, Residues and duality. Lecture Notes in Mathematics, \textbf{20}. \textit{Springer-Verlag, Berlin-New York} (1966).
\bibitem[Ha2]{Ha2} R. Hartshorne, Local cohomology. Lecture Notes in Mathematics, \textbf{41}. \textit{Springer-Verlag, Berlin-New York} (1967).
\bibitem[HTT]{HTT} R. Hotta, K. Takeuchi, T. Tanisaki, $\DD$-modules, perverse sheaves, and representation theory. Progress in Mathematics, \textbf{236}. \textit{Birkhäuser Boston, Inc., Boston, MA} (2008).
\bibitem[Iv]{Iv} B. Iversen, Cohomology of Sheaves. Universitext. \textit{Springer-Verlag, Berlin} (1986).
\bibitem[Ly]{Ly} G. Lyubeznik, On some local cohomology modules. \textit{Adv. Math.} \textbf{213} (2007), no. 2, 621-643.
\bibitem[Me]{Me} Z. Mebkhout, Le formalisme des six opérations de Grothendieck pour les $\DD_X$-modules cohérents. Travaux en Cours, \textbf{35}. \textit{Hermann, Paris} (1989).
\bibitem[OS]{OS} P. Orlik, L. Solomon, Combinatorics and topology of complements of hyperplanes. \textit{Invent. Math.} \textbf{56} (1980), no. 2, 167-189.
\bibitem[Ro]{Ro} J. Rotman, An introduction to homological algebra. Pure and Applied Mathematics, \textbf{85}. \textit{Academic Press, Inc., New York-London} (1979).
\bibitem[We]{We} C. Weibel, An introduction to homological algebra. Cambridge Studies in Advanced Mathematics, \textbf{38}. \textit{Cambridge University Press, Cambridge} (1994).
\bibitem[EGA III]{EGA} A. Grothendieck, Eléments de géométrie algébrique. III. Étude cohomologique des faisceaux cohérents. I. \textit{Inst. Hautes Études Sci. Publ. Math.} \textbf{11} (1961).
\bibitem[SGA 4 1/2]{SGA4} P. Deligne, Cohomologie étale. Séminaire de Géométrie Algébrique du Bois-Marie SGA 4 1/2. Avec la collaboration de J. F. Boutot, A. Grothendieck, L. Illusie et J. L. Verdier. Lecture Notes in Mathematics, \textbf{569}. \textit{Springer-Verlag, Berlin-New York} (1977).
\end{thebibliography}
\end{document}